\documentclass[16pt]{article}

\usepackage{authblk}
\usepackage[english]{babel}
\usepackage[utf8x]{inputenc}
\usepackage[T1]{fontenc}
\usepackage{setspace}
\usepackage{fancyhdr}

\usepackage[a4paper,top=1in,bottom=1in,left=1in,right=1in,marginparwidth=1.75cm]{geometry}
\usepackage{amsmath}
\usepackage{amsthm}
\usepackage{amssymb}

\usepackage{color}
\usepackage{graphicx}
\usepackage[colorinlistoftodos]{todonotes}
\usepackage[colorlinks=true, allcolors=blue]{hyperref}
\usepackage{fancyhdr}
\usepackage{listings}
\usepackage{chngcntr}
\usepackage[ruled, 
lined, 
commentsnumbered]{algorithm2e}

\newtheorem{theorem}{Theorem}
\newtheorem{conjecture}{Conjecture}
\newtheorem{lemma}{Lemma}
\newtheorem{corollary}{Corollary}
\theoremstyle{definition}
\newtheorem{definition}{Definition}

\counterwithout{equation}{section}

\title{Bimonotone Subdivisions of Point Configurations in the Plane}
\date{\today}
\author[1,2]{Elina Robeva} 
\author[1]{Melinda Sun}
\affil[1]{Massachusetts Institute of Technology}
\affil[2]{The University of British Columbia}

\begin{document}
\maketitle

\begin{abstract}
Bimonotone subdivisions in two dimensions are subdivisions all of whose sides are vertical or have nonnegative slope. They correspond to statistical estimates of probability distributions of strongly positively dependent random variables. The number of bimonotone subdivisions compared to the total number of subdivisions of a point configuration provides insight into how often the random variables are positively dependent. We give recursions as well as formulas for the numbers of bimonotone and total subdivisions of $2\times n$ grid configurations in the plane. Furthermore, we connect the former to the large Schröder numbers. We also show that the numbers of bimonotone and total subdivisions of a $2\times n$ grid are asymptotically equal. We then provide algorithms for counting bimonotone subdivisions for any $m \times n$ grid. Finally, we prove that all bimonotone triangulations of an $m \times n$ grid are connected by flips. This gives rise to an algorithm for counting the number of bimonotone (and total) triangulations of an $m\times n$ grid.
\end{abstract}

\section{Introduction}
In this paper we study bimonotone subdivisions in the plane, which are intimately related to nonparametric density estimation (cf. Section~\ref{sec:motivation}). Let $\mathcal A = \{a_1,\ldots, a_m\}\subset \mathbb R^2$ be a {\em point configuration} in the plane. A {\em subdivision} of $\mathcal A$ is a collection of convex polygons whose vertices lie in $\mathcal A$ such that the union of the polygons is the convex hull of $A$ and each pair of polygons either does not intersect or intersects at a common vertex or side. A \textit{triangulation} of $\mathcal A$ is a subdivision of $\mathcal A$ for which all polygons are triangles.
For example, in Figure~\ref{fig:subdivisions}, the leftmost and rightmost drawings are subdivisions, and the rightmost is also a triangulation. The second is not a subdivision because two distinct polygons intersect in their interiors, and the third is not a subdivision because one of the polygons is not convex. Note that not all points in $\mathcal A$ need to be used as vertices of the polygons in the subdivision/triangulation. For more details on subdivisions and triangulations, please refer to the textbook~\cite{triangulations}.

\begin{figure}[h]
\centering
\includegraphics[width=0.8\textwidth]{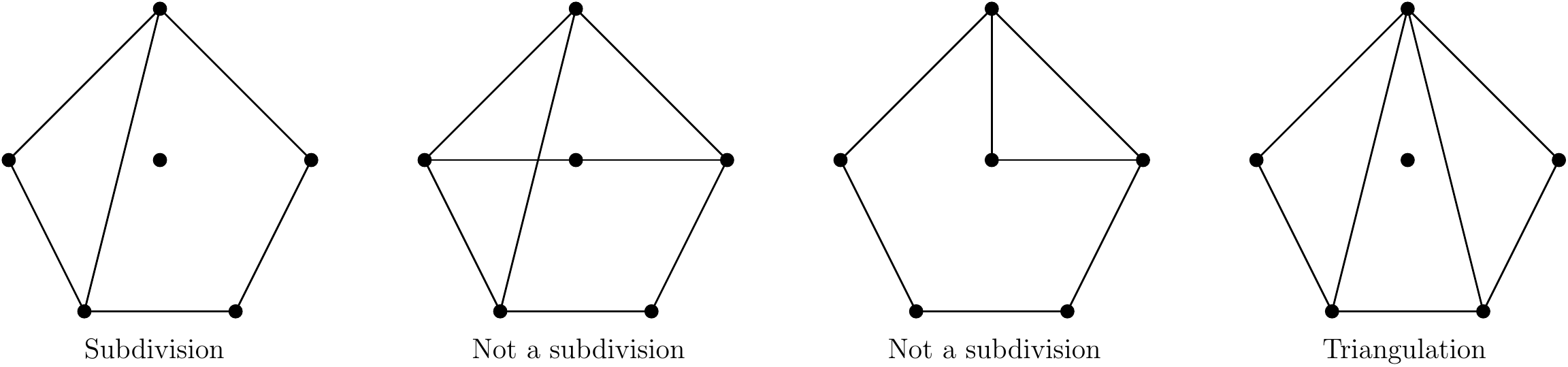}
\caption{\label{fig:subdivisions}Examples of subdivisions and non-subdivisions.}
\end{figure}

\begin{definition}
A \textit{bimonotone polygon} is a polygon for which all edges have either vertical or nonnegative slope.
A \textit{bimonotone subdivision} is a subdivision for which all component polygons of the subdivision are bimonotone.
\end{definition}

\noindent For example, Figure~\ref{fig:bimonotone} below shows two subdivisions, one of which is bimonotone. Bimonotone polytopes are precisely those convex polytopes that are closed under taking coordinate-wise minima and maxima of pairs of points. They were studied in the 1970's by George Bergman and discussed in \cite{BP75} and \cite{Top76}. 
Bimonotone polytopes were later studied in computer science~\cite{QT06} and in discrete geometry under the name {\em distributive}~\cite{FelKna}. 

\begin{figure}[h]
\centering
\includegraphics[width=0.4\textwidth]{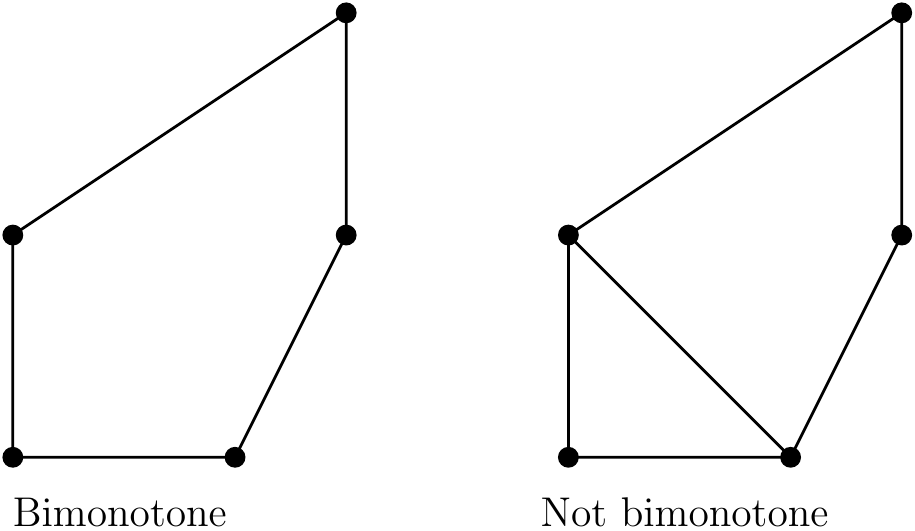}
\caption{\label{fig:bimonotone}Examples of bimonotone and non-bimonotone subdivisions.}
\end{figure}
\bigskip

\subsection{Motivation.}\label{sec:motivation}
For a point configuration $\mathcal A=\{a_1,\ldots, a_m\}$ and a set of heights, or tent poles, $\{h_1,\ldots, h_m\}$, one above each of the points in $\mathcal A$, we can define a {\em tent function} as the smallest concave function whose value at $a_i$ is at least as big as $h_i$ for each $i=1, \ldots, m$. In other words, the tent function is formed by spreading a "tarp" over the poles, see Figure~\ref{fig:tentfunction}.
Each tent function induces a subdivision of $\mathcal A$ composed of the polygons in the plane above which the tent function is linear. In fact, the subdivisions of $\mathcal A$ that arise from a tent function are called {\em regular subdivisions}. For most point configurations, there exist subdivisions that are not regular~\cite{triangulations}.

\begin{figure}[h]
\centering
\includegraphics[width=0.5\textwidth]{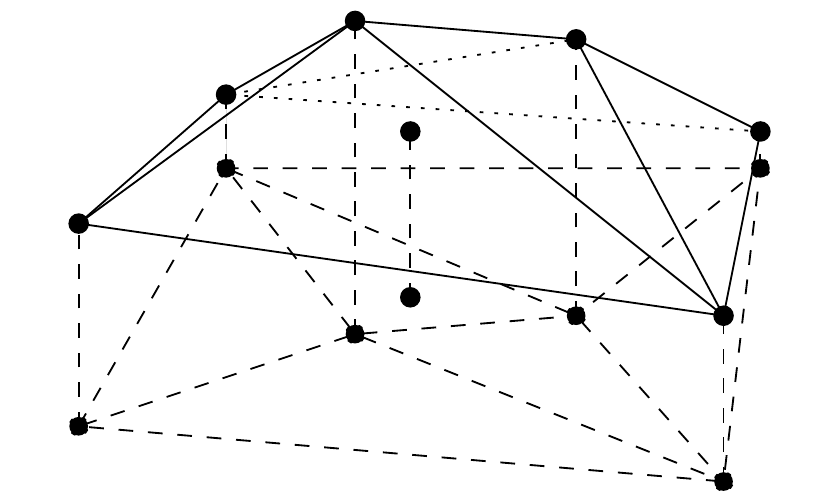}
\caption{\label{fig:tentfunction}Example of a tent function from \cite{maxlikelihood}.}
\end{figure}

Tent functions show up in the field of nonparametric statistics~\cite{Wasserman} as the estimates of log-concave functions~\cite{maxlikelihood}. (A function is log-concave if its logarithm is concave.) More precisely, if the points in $\mathcal A$ are the samples drawn from an unknown log-concave density $p$, then, the maximum likelihood estimate of $p$ will be $\hat p = \exp(f)$, where $f$ is a tent function with tent poles centered at the samples $\mathcal A$~\cite{maxlikelihood}. It was recently shown that if the unknown density $\hat p$ is log-concave and log-supermodular (also known as multivariate totally positive of order 2, or MTP$_2$, cf. Definition~\ref{def:MTP2}), then the maximum likelihood estimate of $p$ is a density $\hat p=\exp(f)$, where $f$ is a tent function which induces a bimonotone subdivision of the point configuration $\mathcal A$.

\begin{definition}\label{def:MTP2}
A function $f:\mathbb R^d \to\mathbb R\cup\{-\infty\}$ is \textit{supermodular} if $f(x) + f(y) \leq f(\min(x,y)) + f(\max(x,y))$ for all $x,y$ in the domain of $f$. A density $p$ is {\em log-supermodular}, or {\em multivariate totally positive of order 2} ({\em MTP$_2$}) if $p=\exp(f)$, where $f$ is supermodular.
\end{definition}

If the density $p$ of a random vector $X = (X_1,\ldots, X_d)$ is MTP$_2$, then, the coordinates of $X$ are strongly positively dependent on each other. In fact, MTP$_2$ implies another strong form of positive dependence called {\em positive association}~\cite{FKG}. Despite it being such a strong form of dependence, the MTP$_2$ property holds for a variety of real-world distributions as well as many well-studied families of distributions~\cite{BF00, FLSUWZ, Fel73, SH15, LUZ17}.

For this reason, the authors of~\cite{maxlikelihood2} study the problem of estimating a log-concave and MTP$_2$ density. They show that if the samples lie in $\mathbb R^2$, the maximum likelihood estimate $\hat p$ equals $\exp(f)$, where $f$ is a  tent function, which induces a bimonotone subdivision~\cite{maxlikelihood2}. One of the important remaining questions is that of characterizing how large the family of log-concave and MTP$_2$ densities is, especially compared to the family of log-concave densities. Knowing this would shed light on the statistical complexity of the problem of estimating log-concave and MTP$_2$ densities.. If the log-concave and MTP$_2$ family is much smaller, then its statistical complexity should intuitively be much better. We here show that in two dimensions the two families are asymptotically the same in size (cf. Theorems~\ref{thm:1} and \ref{thm:2}). These findings are consistent with the work. We leave the same computation in dimension $d\geq 3$ to future work, conjecturing that in these cases the MTP$_2$ and log-concave family is much smaller than the log-concave family.

\smallskip

Triangulations of polygons can be counted using the flip graph method~\cite{triangulations}. It is also known that the number of triangulations of an $n$-sided polygon is equal to the $(n-2)$th Catalan number~\cite{triangulations}. However, little research has been conducted on the number of subdivisions and triangulations of grids, or of bimonotone subdivsions and triangulations of any point configuration.
\smallskip
\smallskip

\subsection{Organization of the paper.} 
In Section~\ref{sec:recursions}, we study point configurations $\mathcal A\subset \mathbb R^2$ whose points lie on two rows of a rectangular grid. We provide a recursion for the number of bimonotone and total subdivisions of grids with two rows of possibly different numbers of points. We use these recursins in Section~\ref{sec:general_form} to find a general formula for these numbers. Furthermore, we show that the number of bimonotone subdivisions of a $2 \times n$ grid is equal to $2^{n-2}$ multiplied by the $n$th Schröder number. In Section~\ref{sec:bim_triangulations} we show that bimonotone triangulations using all of the vertices vertices of a grid configuration $\mathcal A$ form a connected flip graph. Using this result, in Section~\ref{sec:algorithms}, we present two algorithms for counting the numbers of bimonotone and total subdivisions and triangulations of two-dimensional grids. In Section~\ref{sec:conclusion} we conclude with further research questions.

\section{Recursions}\label{sec:recursions}

In this section we derive recursions for the number of bimonotone subdivisions and total subdivisions for grids consisting of two rows. Let $P_{m,n}$ denote a grid with 2 rows that has $m$ points in the top row and $n$ points in the bottom, aligned at the left, and let the bottom left point be at the origin. Let $B_{m,n}$ be the number of bimonotone subdivisions of this configuration. We set up a recursion to count $B_{m,n}$.

\begin{lemma}
The number of bimonotone subdivisions $B_{m,n}$ of $P_{m,n}$ is $2B_{m,n-1} + 2B_{m-1,n} - 2B_{m-1,n-1}$ if $m>n$, $2B_{m,n-1}$ if $m=n$, and $0$ if $m<n$.
\end{lemma}

\begin{figure}[h]
\centering
\includegraphics[width=0.4\textwidth]{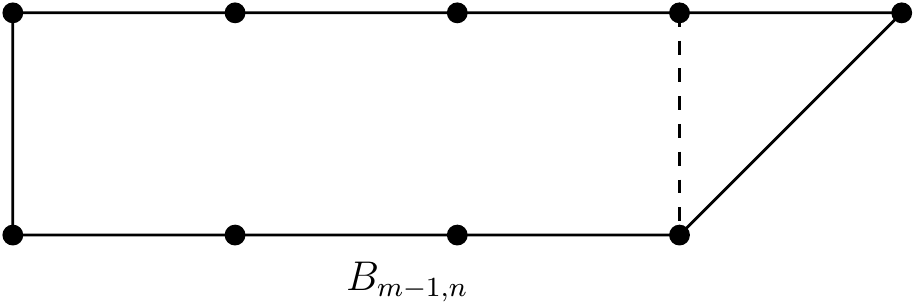}
\caption{\label{fig:case1}Top right vertex unconnected.}
\end{figure}

\begin{proof}
For $m>n$, if the top right vertex $(m-1,1)$ is not connected to any vertex other than its left neighbor $(m-2,1)$ and the bottom right vertex $(n-1,0)$, as shown in Figure~\ref{fig:case1}, then there are $2B_{m-1,n}$ bimonotone subdivisions. This is because if we pair each bimonotone subdivision with the subdivision that has the edge connecting $(n-1,0)$ and $(m-2,1)$ toggled, then each pair corresponds to the unique bimonotone subdivision of $P_{m-1,n}$ with the same internal edges.

\begin{figure}[h]
\centering
\includegraphics[width=0.4\textwidth]{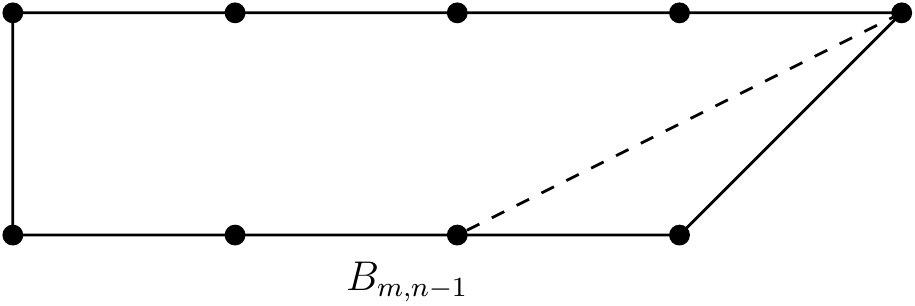}
\caption{\label{fig:case2}Bottom right vertex unconnected.}
\end{figure}

By the same reasoning, there are $2B_{m,n-1}$ bimonotone subdivisions when the bottom right vertex $(n-1,0)$ is not connected to any points but $(n-2,0)$ and $(m-1,1)$, as shown in Figure~\ref{fig:case2}.

\begin{figure}[h]
\centering
\includegraphics[width=0.4\textwidth]{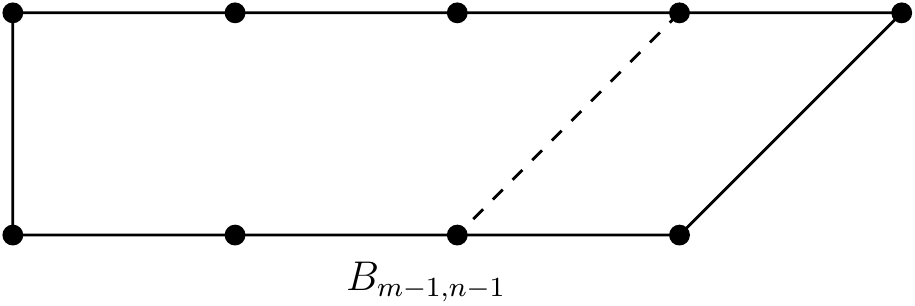}
\caption{\label{fig:case3}Top and bottom right vertices unconnected.}
\end{figure}

It is not possible for both vertices to be connected to points other than their external neighbors, as the edges would intersect at a point not in the configuration. And when both are not connected, as in Figure~\ref{fig:case3}, there are $2B_{m-1,n-1}$ bimonotone subdivisions because similarly to above, the bimonotone subdivisions correspond to bimonotone subdivisions of $P_{m-1,n-1}$. So subtracting the overlap, $B_{m,n} = 2B_{m,n-1} + 2B_{m-1,n} - 2B_{m-1,n-1}$.

When $m<n$, there are 0 bimonotone subdivisions because the edge connecting $(m-1,1)$ and $(n-1,0)$ has a negative slope.

\begin{figure}[h]
\centering
\includegraphics[width=0.3\textwidth]{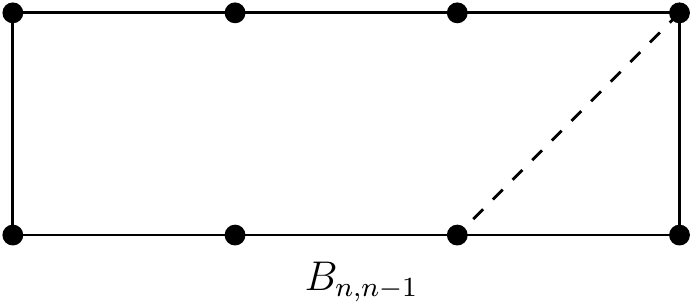}
\caption{\label{fig:case4}$m=n$, bottom right vertex unconnected.}
\end{figure}

For $m=n$, as in Figure~\ref{fig:case4}, it is impossible for the bottom right vertex to be connected to any point but its neighbors as the slope would be negative. So just as in the first case of $m>n$, $B_{m,n} = 2B_{m,n-1}$.

In summary,
\[B_{m,n} =
\begin{cases} 
2B_{m,n-1} + 2B_{m-1,n} - 2B_{m-1,n-1} & m>n \\
2B_{m,n-1} & m=n \\
0 & m<n
\end{cases}.\]
\end{proof}

Formulas for $B_{m,n}$ for small fixed $n$ can be found through the recursion. These are shown in Table~\ref{tab:bimonotonesub}.

\begin{table}[h]
\centering
\begin{tabular}{l|l}
$n$ & $B_{m,n}$ \\\hline
1 & $2^{m-2}$ \\
2 & $2^{m-2}(m)$ \\
3 & $\frac{2^{m-2}}{2}(m^2+3m-6)$ \\
4 & $\frac{2^{m-2}}{6}(m^3+9m^2-4m-60)$ \\
5 & $\frac{2^{m-2}}{24}(m^4+18m^3+47m^2-258m-600)$
\end{tabular}
\caption{\label{tab:bimonotonesub}Number of Bimonotone Subdivisions of $P_{m,n}$.}
\end{table}

Now let $A_{m,n}$ be the the number of subdivisions, not necessarily bimonotone, of $P_{m,n}$.

\begin{lemma}
The number of subdivisions of $P_{m,n}$ is $A_{m,n} = 2A_{m,n-1} + 2A_{m-1,n} - 2A_{m-1,n-1}$.
\end{lemma}

\begin{proof}
We again consider the connectivity of the top right and bottom right vertices and use the inclusion-exclusion principle. Using the same reasoning as for the $m>n$ case for $B_{m,n}$, we get the same recursion.
\end{proof}

Expressions for $A_{m,n}$ can now be found similarly, as shown in Table~\ref{tab:sub}.
\begin{table}[h]
\centering
\begin{tabular}{l|l}
$n$ & $A_{m,n}$ \\\hline
1 & $2^{m-2}$ \\
2 & $2^{m-2}(m+1)$ \\
3 & $\frac{2^{m-2}}{2}(m^2+5m+2)$ \\
4 & $\frac{2^{m-2}}{6}(m^3+12m^2+29m+6)$ \\
5 & $\frac{2^{m-2}}{24}(m^4+22m^3+131m^2+206m+24)$
\end{tabular}
\caption{\label{tab:sub}Number of Subdivisions of $P_{m,n}$.}
\end{table}

\section{General Form}\label{sec:general_form}
Next, we use the recursions from Section~\ref{sec:recursions} to find the general forms of the numbers of bimonotone and total subdivisions of $P_{m,n}$. We begin with bimonotone subdivisions.

\begin{theorem}\label{thm:1}
The number of bimonotone subdivisions of $P_{m,n}$ is given by $B_{m,n} = \frac{2^{m-2}}{(n-1)!}P_n(m)$, where $P_n(m)$ is some monic polynomial with degree $n-1$.
\label{thm:genbimonotone}
\end{theorem}

\begin{proof}
We use induction. First, $B_{m,1} = 2^{m-2}$, because for each top vertex except $(0,1)$ and $(m-1,1)$, the vertex may connect to the bottom vertex or may not. Now we consider $B_{m,n}$ in terms of $P_{n-1}(m)$. From the recursion,
\[B_{m,n} = \frac{2^{m-2}}{(n-2)!} \left(2P_{n-1}(m)-P_{n-1}(m-1)\right) + 2B_{m-1,n}. \]
Plugging in $B_{m-1,n}$ and so on gives
\begin{align*}
B_{m,n} &= \frac{2^{m-2}}{(n-2)!} \left(2P_{n-1}(m) + \sum_{i=n+1}^{m-1}\left(-P_{n-1}(i)+2P_{n-1}(i)\right) -P_{n-1}(n)\right) + 2^{m-n}B_{n,n} \\ 
& = \frac{2^{m-2}}{(n-2)!} \left(2P_{n-1}(m) + \sum_{i=n}^{m-1}-P_{n-1}(i)+2P_{n-1}(i)\right) \\
&= \frac{2^{m-2}}{(n-2)!} \left(P_{n-1}(m)+ \sum_{i=n}^m P_{n-1}(i)\right).
\end{align*}

Let $S(m,p)$ be the sum of the $p$th powers of the first $m$ positive integers. Faulhaber's formula gives this as
\[S(m,p) = \sum_{k=1}^m k^{p} = \frac{m^{p+1}}{p+1}+\frac{1}{2}m^p+\sum_{k=2}^p \frac{B_{k}}{k!}p^{\underline{k-1}}m^{p-k+1} \]
where $p^{\underline{k-1}}=p!/[p-(k-1)]!$ and the $B_k$ are the Bernoulli numbers \cite{faulhaber}.

Let $P_{n-1}(m) = m^{n-2} + \sum_{i=0}^{n-3} a_i m^i$. Then from Faulhaber's formula,
\[B_{m,n} = \frac{2^{m-2}}{(n-2)!} \left(P_{n-1}(m) + (S(m,n-2) - S(n,n-2))  + \sum_{i=0}^{n-3} a_i(S(m,i)-S(n,i))\right).\]

Only $S(m,n-2)$ contains a $m^{n-1}$ term, which is $\frac{m^{n+1}}{(n+1)}$ from Faulhaber's formula. None of the other terms contain a higher degree term. Thus the $m^{n-1}$ term of $B_{m,n}$ is $\frac{2^{m-2}}{(n-1)!} m^{n-1}$. Since all other terms are polynomial, this is a polynomial.

Thus $B_{m,n}$ can be expressed in the form $B_{m,n} = \frac{2^{m-2}}{(n-1)!}P_n(m)$ where $P_n(m)$ is monic and of degree $n-1$.
\end{proof}

We prove a similar result for the total number of subdivisions.

\begin{theorem}\label{thm:2}
The number of subdivisions of $P_{m,n}$ is given by $A_{m,n} = \frac{2^{m-2}}{(n-1)!}Q_n(m)$, where $Q_n(m)$ is some monic polynomial of degree $n-1$.
\end{theorem}

\begin{proof}
We again use induction. Exactly as for bimonotone subdivisions, we find $A_{m,1} = 2^{m-2}$, satisfying the base. Then,  considering $A_{m,n}$ in terms of $Q_{n-1}(m)$,
\[A_{m,n} = \frac{2^{m-2}}{(n-2)!} (2Q_{n-1}(m)-Q_{n-1}(m-1)) + 2A_{m-1,n}. \]
Plugging in $A_{m-1,n}$ and so on gives
\begin{align*}
A_{m,n} &= \frac{2^{m-2}}{(n-2)!} \left(2Q_{n-1}(m)+\sum_{i=1}^{m-1}-Q_{n-1}(i)+2Q_{n-1}(i)\right) \\ 
&= \frac{2^{m-2}}{(n-2)!} \left(Q_{n-1}(m)+\sum_{i=1}^{m-1}Q_{n-1}(i)\right).
\end{align*}

This proceeds in the same way as the proof of Theorem~\ref{thm:genbimonotone} without the $-S(n,n)$ term, which does not affect the result. Thus $A_{m,n}$ can be expressed in the form $A_{m,n} = \frac{2^{m-2}}{(n-1)!}Q_n(m)$ where $Q_n(m)$ is monic and of degree $n-1$.
\end{proof}

Note that the two forms are (asymptotically) identical.

\begin{corollary}
The number of bimonotone subdivisions of $P_{m,n}$ is asymptotically equivalent to the total number of subdivisions for large values of $m$.
\end{corollary}

Now, we connect bimonotone subdivisions to the large Schröder numbers~\cite{faulhaber}. The $n$th large Schröder number $S_{n}$ is the number of paths from $(0,0)$ to $(n,n)$ where unit steps can be taken north, east, or northeast and no points on the path lie above the line $y=x$.

\begin{theorem}
The number of bimonotone subdivisions of the $2 \times n$ lattice grid $P_{m,n}$ is equal to $2^{n-2}$ multiplied by the $(n-1)$th large Schröder number.
\label{thm:schroder}
\end{theorem}

\begin{proof}
We strongly induct on $n$. The large Schröder numbers are known to follow the relation $S_n = S_{n-1} + \sum_{k=0}^{n-1} S_k S_{n-1-k}$. First, $B_{1}$, the number of bimonotone subdivisions of a $2 \times 1$ grid, has little meaning, so we define it to be $\frac{1}{2}$ to fit the relationship with the Schröder numbers.

For a $2 \times n$ grid, the bimonotone subdivisions can be divided into the cases where the are no internal vertical edges, or where the leftmost vertical edge occurs at $x=k+1$.

\begin{figure}[h]
\centering
\includegraphics[width=0.4\textwidth]{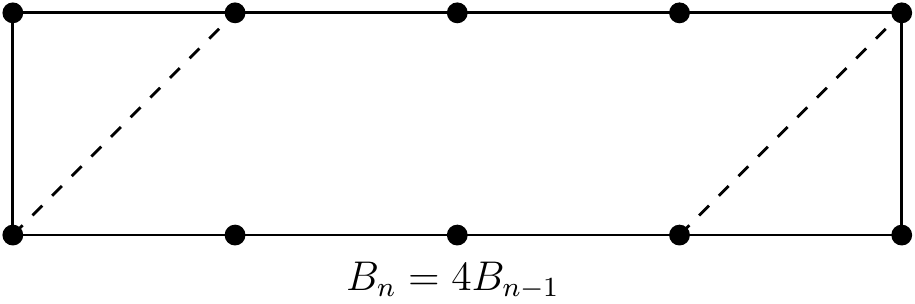}
\caption{\label{fig:schroeder1}No vertical edges.}
\end{figure}

The first case, shown in Figure~\ref{fig:schroeder1}, corresponds to the bimonotone subdivisions of a $2 \times (n-1)$ grid. We can see this by considering the $2 \times (n-1)$'s bimonotone subdivisions after shifting the top vertices of the grid one unit to the right. The shifted subdivision cannot contain vertical lines, as that would make the original not bimonotone. These shifted bimonotone subdivisions are also all those of the original $2 \times n$ grid, except the leftmost and rightmost edges of the $2 \times (n-1)$ grid can be present or not in the $2 \times n$ grid. This makes the number of bimonotone subdivisions in this case $4B_{n-1}$.

\begin{figure}[h]
\centering
\includegraphics[width=0.4\textwidth]{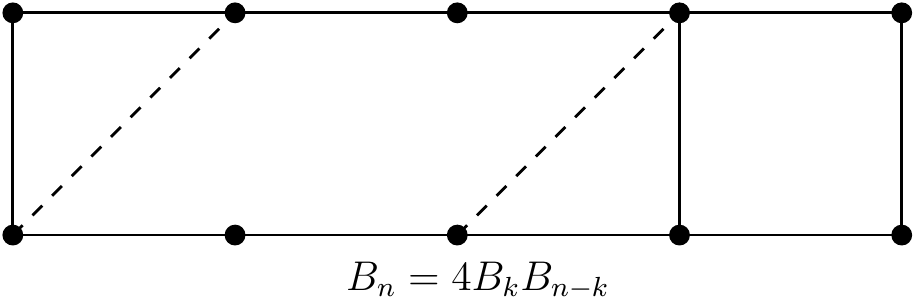}
\caption{\label{fig:schroeder2}Leftmost vertical edge at $x=k+1$.}
\end{figure}

If the leftmost internal vertical edge occurs at $x=k+1$, as shown in Figure~\ref{fig:schroeder2}, then we can separately consider the $2 \times (k+1)$ grid to the left and $2 \times (n-k)$ to the right. By the same reasoning as in the first case, the number of bimonotone subdivisions of the left side, which has no vertical edges, is $4B_{k}$. Note that for $k=1$, there are 2 bimonotone subdivisions of the left side, which agrees with our definition of $B_{1}$ as $\frac{1}{2}$. The number of bimonotone subdivisions of the right side is simply $B_{n-k}$. Thus the total is $4B_{k}B_{n-k}$.

Adding these, placing $2B_{n-1}$ into the summation as $4B_{1}B_{n-1}$, and plugging in $B_{k} = 2^{k-2}S_{k-1}$ for $k<n$, we get
\begin{align*}
B_n &= 2B_{n-1} + 4\sum_{k=1}^{n-1} B_k B_{n-k} \\
&= 2^{n-2} \left(S_{n-2} + \sum_{k=0}^{n-2} S_k S_{n-2-k}\right) \\
&= 2^{n-2}S_{n-1}.
\end{align*}
\end{proof}

\section{Bimonotone triangulations}\label{sec:bim_triangulations}
In this section, we consider the more specific problem of counting bimonotone triangulations. Because triangulations can be counted using the flip method~\cite{triangulations}, we apply the method to the bimonotone case.

\begin{definition}
A \textit{flip} takes a quadrilateral in a triangulation and switches which of its two diagonals is included in the triangulation.
\end{definition}

\begin{theorem}
Every bimonotone triangulation using all vertices of an equally spaced $m \times n$ lattice grid can be flipped to every other bimonotone triangulation of the grid.
\end{theorem}

\begin{proof}
We will show that every bimonotone triangulation can be flipped to the triangulation with every vertical edge and positive-slope diagonal $(i,j)-(i+1,j+1)$ present (cf. Figure~\ref{fig:main}).
\smallskip

\begin{figure}[h]
\centering
\includegraphics[width=0.35\textwidth]{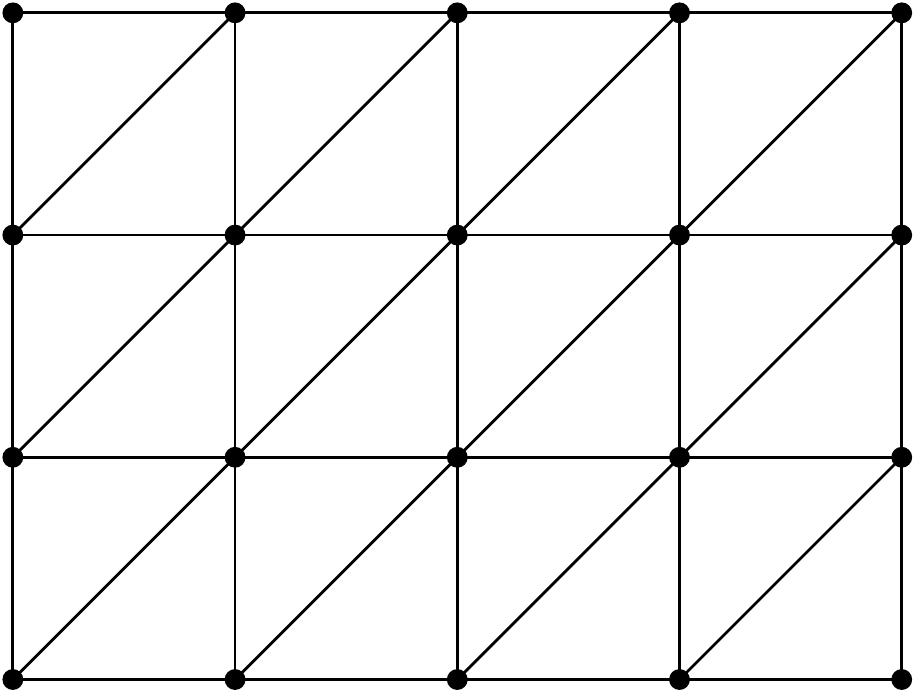}
\caption{\label{fig:main} Every bimonotone triangulation that uses all grid points can be flipped to this triangulation.}
\end{figure}

Take the longest diagonal of a particular triangulation. Let this range from $(0,0)$ to $(x_d, y_d)$. We know that $x_d$ and $y_d$ are relatively prime, as the diagonal would intersect a lattice point otherwise.

We will first prove that the two vertices completing the two triangles using the diagonal must be the two points closest to the edge on either side. This will allow us to prove that the quadrilateral that has this diagonal can be flipped and remain bimonotone.

Consider the closest points below the line for each value of $x$. If $s$ is the slope of the line, then these points' vertical distances from the diagonal are the fractional parts $\{s\}, \{2s\}, \dots, \{(x_d-1)s\}$. As each of these has numerator $ky_d \mod{x_d}$ and $y_d$ is relatively prime to $x_d$, this is an ordering of $\frac{1}{x_d}, \frac{2}{x_d}, \dots, \frac{x_d-1}{x_d}$.

Let the closest point be at $(x_c,y_c)$. The vertical distance from this point to the diagonal is $\frac{1}{x_d}$. We must prove that no lattice points lie strictly within the two triangles bounded by the extensions of the two lines connecting the closest vertex and one of the diagonal's endpoints (cf. the grey triangles in Figure~\ref{fig:triangulation}). Otherwise, a triangulation could use a point in this region as the third vertex instead of the closest point. The greatest vertical distance within the region is $\frac{1}{x_c}$ for the right triangle and $\frac{1}{x_d} \cdot \frac{y_d}{y_d-y_c}$ for the left triangle, by similar triangles.

\begin{figure}[h]
\centering
\includegraphics[width=0.5\textwidth]{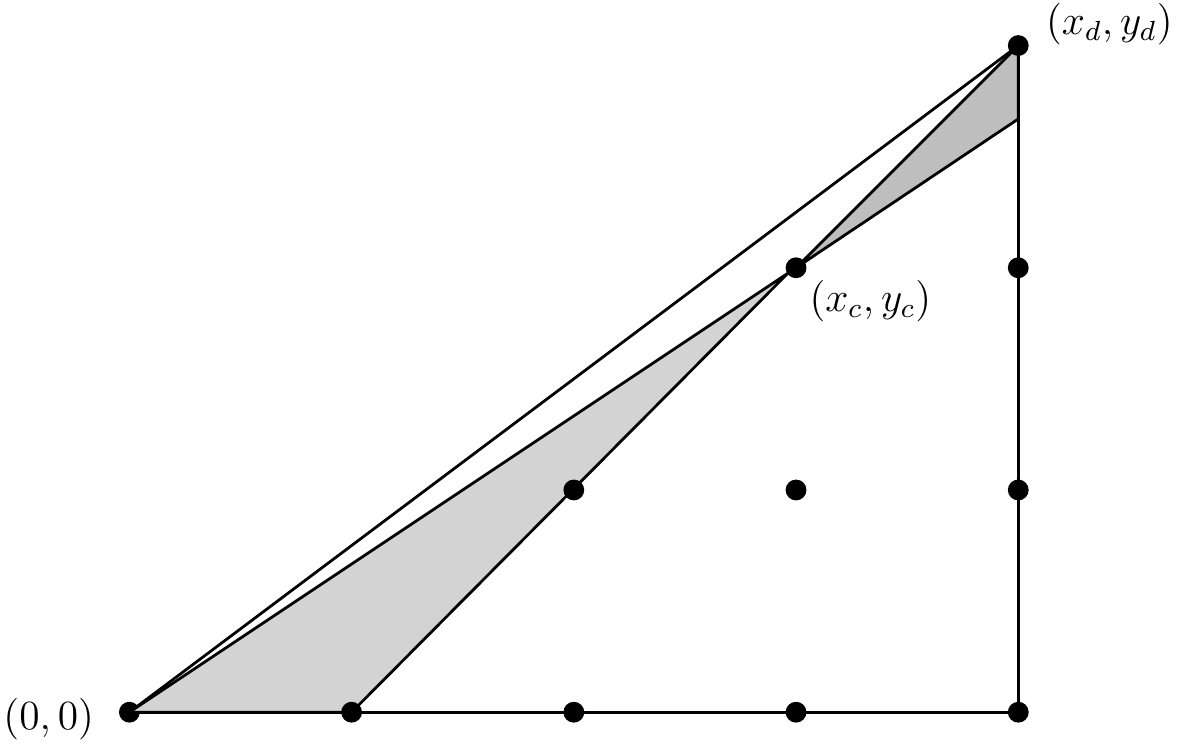}
\caption{\label{fig:triangulation}No points can lie within the two shaded triangles.}
\end{figure}

For each integer $2 \leq k \leq x_d-1$, we ensure the point with distance $\frac{k}{x_d}$ does not lie within the two triangles. For now, only consider whether the point is in the triangle on the right. Either $x_c < \frac{x_d}{k}$ or $x_c \geq \frac{x_d}{k}$. If $x_c < \frac{x_d}{k}$, then by similar triangles the point with distance $\frac{ky_d}{x_d}$ lies on the line connecting the left vertex of the diagonal and the closest point, to the right of the closest point. Thus it is not strictly within the triangle. If $\frac{x_d}{k} \leq x_c$, then $\frac{k}{x_d} \geq \frac{1}{x_c}$, the maximum vertical distance in the triangle, so the point cannot lie within the triangle. By symmetry, the same is true for the left triangle. Therefore, only the closest point on one side of the diagonal can be the third vertex for a triangle using it.

By symmetry, the closest points on either side of the diagonal are in mirrored positions, so that the quadrilateral formed is a parallelogram. We will now prove that when this quadrilateral is flipped, its new diagonal is shorter and remains bimonotone.

\begin{figure}[h]
\centering
\includegraphics[width=0.5\textwidth]{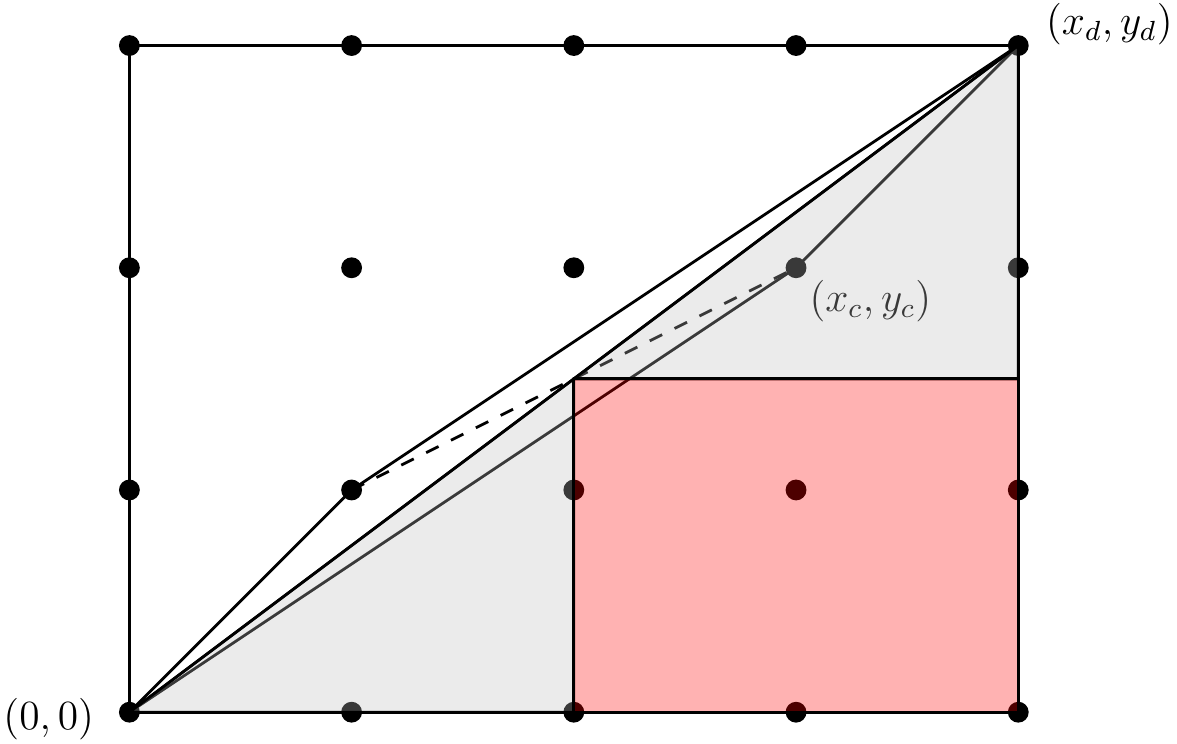}
\caption{\label{fig:flip}The closest point must lie in the two gray triangles, not the red rectangle.}
\end{figure}

Consider the vertex $(x_c,y_c)$ of the triangle below the diagonal. For the flipped diagonal to be bimonotone, the vertex must be either above or to the left of the midpoint of the diagonal (nonstrictly),i.e., the grey regions in Figure~\ref{fig:flip}. This condition is always met. If $s$ is the slope of the diagonal, the closest point to the diagonal in the region strictly below and to the right of the midpoint, i.e., the red region in Figure~\ref{fig:flip}, is either $\frac{1}{2} + \frac{s}{2}$ or $1 + \frac{s}{2}$ or $\frac{1}{2} + s$ vertically away from the diagonal. There are guaranteed to be points closer than $\frac{1}{2}$ in the desired (grey) regions: By symmetry, if the closest point in the $i$th column of points is $t$ vertically below the diagonal, then the closest point in the $(x_d-i+1)$th row is $1-t$ below. So there exist points less than or equal to $\frac{1}{2}$ vertically below the diagonal, and, therefore, the closest point to the diagonal has to be in the desired (grey) regions. Therefore, we can flip the diagonal to replace it with one that is still bimonotone but shorter.

This can be repeated for all diagonals until the configuration with only horizontal and vertical edges and unit diagonals is reached. Therefore, all bimonotone triangulations of an $m \times n$ grid are connected by flips.
\end{proof}

\section{Algorithms}\label{sec:algorithms}

We now present an algorithm to count all subdivisions and all bimonotone subdivisions of an $m \times n$ grid. 
The algorithm considers all possible internal edges between the vertices of the grid. To count bimonotone subdivisions, only edges of nonnegative or vertical slope are considered.
All possible combinations of edges are then tested for being a subdivision. For each pair of edges, we check if they intersect within the grid outside a vertex. If none do, then we check convexity: For each vertex in the interior of the grid, the angles between consecutive edges protruding out of the vertex must be less than or equal to $\pi$. Successful subdivisions are added to the count until all are counted.
\smallskip
\smallskip

	\begin{algorithm}[H]
		\label{alg:QT}
		\caption{Count all subdivisions and all bimonotone subdivisions of an $m\times n$ grid.}
		\LinesNumbered
		\DontPrintSemicolon
		\SetAlgoLined
		\SetKwInOut{Input}{Input}
		\SetKwInOut{Output}{Output}
		\Input{The sizes $m$ and $n$ of an $m\times n$ grid.}
		\Output{The numbers of all subdivisions and all bimonotone subdivisions of the grid.}
		\BlankLine
		\For{All subsets of (bimonotone) edges between the vertices of the $m\times n$ grid}{
		\If{Any two edges intersect at a point not on the grid}{\textbf{Continue.}}}
		\For{Each vertex from the grid with an edge adjacent to it}{\If{The angles between consecutive edges protruding out of the vertex are at most $\pi$}{Add 1 to the count.}}
		\Output{The numbers of bimonotone and all subdivisions.}
	\end{algorithm}

\smallskip
\smallskip
The output of the algorithm for small values of $m,n$ is shown in Table~\ref{tab:2byn} and Table~\ref{tab:3byn}.

\begin{table}[h]
\centering
\begin{tabular}{c|c|c}
$n$ & $B_{2,n}$ & $A_{2,n}$ \\ \hline
2 & 2 & 3 \\
3 & 12 & 26 \\
4 & 88 & 252 \\
5 & 720 & 2568 \\
6 & 6304 & 26928 
\end{tabular}
\caption{\label{tab:2byn}Number of Subdivisions of $2 \times n$ grids.}
\end{table}

\begin{table}[h]
\centering
\begin{tabular}{c|c|c}
$n$ & $B_{3,n}$ & $A_{3,n}$ \\ \hline
2 & 12 & 26 \\
3 & 528 & 2224 \\
4 & 34152 &
\end{tabular}
\caption{\label{tab:3byn}Number of Subdivisions of $3 \times n$ grids.}
\end{table}

We now present a second algorithm to count the number of bimonotone triangulations of an $m \times n$ grid for which every point in the grid is a vertex of a triangle. This is a breadth-first search algorithm using flips. Starting from the arrangement including every vertical, horizontal, and positively sloped unit diagonal edge, quadrilaterals can be flipped so that the diagonal present is switched. This is attempted for every possible flip, and resulting arrangements are checked if they are different from previous ones. If new, they are added to the list.

\smallskip

	\begin{algorithm}
		\label{alg:QT}
		\caption{Count all bimonotone triangulations of an $m\times n$ grid that use all points in the grid.}
		\LinesNumbered
		\DontPrintSemicolon
		\SetAlgoLined
		\SetKwInOut{Input}{Input}
		\SetKwInOut{Output}{Output}
		\Input{The sizes $m$ and $n$ of an $m\times n$ grid.}
		\Output{The numbers of all bimonotone triangulations that use all points in the grid.}
		\BlankLine
		Let \textbf{S} be a set with one element: the triangulation containging all unit vertical and horizontal edges as well as the positively sloped diagonal edges of the form $(i,j) - (i+1,j+1)$.
		
		Let \textbf{Que} be a queue and add the same triangulation to it.
		
		\While{\textbf{Que} is not empty}{Let \textbf{T} be the next element in \textbf{Que} (remove \textbf{T} from \textbf{Que}).
		
		\For{All possible flips of diagonals in \textbf{T}}{\If{the new triangulation is \underline{not} in \textbf{S}}{Add it to \textbf{S} and to \textbf{Que}}}}
		\Output{The size of S.}
	\end{algorithm}

\smallskip

\section{Future Research}\label{sec:conclusion}

It would be interesting to prove the following conjecture on the total number of subdivisions of a $2 \times n$ grid.

\begin{conjecture}
The number of subdivisions of a $2 \times n$ grid is equal to $2^{n-2}$ multiplied by the $(n-1)$th central Delannoy number.
\end{conjecture}

The central Delannoy numbers~\cite{faulhaber} are related to the large Schröder numbers, which are similarly involved in the expression for the number of bimonotone subdivisions of a $2 \times n$ grid given in Theorem~\ref{thm:schroder}. The $n$th central Delannoy number counts the number of paths from $(0,0)$ to $(n,n)$ where steps can be taken east, north, or northeast. The difference between these and the large Schröder numbers is that these paths can cross the line $y=x$.

\section{Acknowledgements}

The research presented here took place mostly in 2018 via the MIT PRIMES USA program. We are grateful to Dr. Tanya Khovanova, Dr. Slava Gerovitch, and Prof. Pavel Etingof for organizing the program. At the time ER was supported by an NSF postdoctoral fellowship (DMS-170-3821) and is currently supported by a National Sciences and Engineering Research Council of Canada Discovery Grant (DGECR-2020-00338).

\end{document}